\documentclass[12pt]{article}
\usepackage[a4paper, margin=2.3cm]{geometry}
\usepackage{amsmath,amssymb,amsthm}
\usepackage{color}

\newtheorem{theorem}{Theorem}[section]
\newtheorem{corollary}[theorem]{Corollary}
\newtheorem{lemma}[theorem]{Lemma}
\newtheorem*{definition*}{Definition}

\def\F{\mathbb{F}}
\def\Fp{\mathbb{F}_p}

\def\R{\mathbb{R}}

\def\RR{\mathcal{R}}
\def\S{\mathcal{S}}
\newcommand{\I}{\mathcal{I}}

\begin{document}
\title{Three-variable expanding polynomials and higher-dimensional distinct distances}

\author{Thang Pham\and Le Anh Vinh\and Frank de Zeeuw}
\date{}
\maketitle

\begin{abstract}
We determine which quadratic polynomials in three variables are expanders over an arbitrary field $\F$.
More precisely, 
we prove that for a quadratic polynomial $f\in \F[x,y,z]$, which is not of the form $g(h(x)+k(y)+l(z))$,
we have $|f(A\times B\times C)|\gg N^{3/2}$ for any sets $A,B,C\subset \F$ with $|A|=|B|=|C|=N$,
with $N$ not too large compared to the characteristic of $\F$.

We give several applications. 
We use this result for $f=(x-y)^2+z$
to obtain new lower bounds on $|A+A^2|$ and $\max\{|A+A|,|A^2+A^2|\}$, 
and to prove that a Cartesian product $A\times\cdots \times A\subset \F^d$ determines almost $|A|^2$ distinct distances if $|A|$ is not too large.
\end{abstract}

\section{Introduction}\label{sec:intro}
Let $\F$ be an arbitrary field. 
We use the convention that if $\F$ has positive characteristic, 
we denote the characteristic by $p$,
while if $\F$ has characteristic zero, 
we set $p=\infty$. 
Thus, a condition like $N<p^{2/3}$ is restrictive in positive characteristic, but vacuous in characteristic zero.
Although our focus is on finite fields, some of our results may be of interest in characteristic zero,
and this convention lets us state our results in a uniform way.
We use asymptotic notation: 
We write $X\ll Y$ if there exists a constant $C$, independent of any parameters implicit in $X$ and $Y$, such that $X\leq CY$; 
we write $X\gg Y$ if $Y\ll X$; and we write $X\approx Y$ if both $X\ll Y$ and $X\gg Y$.

Our aim is to study the expansion behavior of polynomials, i.e., to determine when the value set of a polynomial on any finite set is significantly larger than the input.
We wish to classify the polynomials that have this expanding property, and then to quantify the expansion.
The following definition captures this property. 

\begin{definition*}
A polynomial $f\in \F[x_1,\ldots,x_k]$ is an \emph{expander} if there are $\alpha>1,\beta>0$ such that for all sets $A_1,\ldots,A_k\subset\F$ of size $N\ll p^\beta$ we have 
\[|f(A_1\times \cdots \times A_k)|\gg N^\alpha.\]
\end{definition*}

Note that other sources may have slightly different definitions of expanders, but the essence is usually the same.
One distinctive aspect is that we allow the sets $A_i$ to be distinct; if one requires them to be the same, one obtains a strictly larger class of polynomials.
Also note that if $A$ is a subfield of $\F$ of size $N$, 
then $|f(A\times \cdots \times A)| = N$, so in positive characteristic we must have $\beta<1$.
In characteristic zero, $\beta$ plays no role.

In the wake of a recent result of Rudnev \cite{R} (see Theorem \ref{thm:rudnev}),
based on work of Guth and Katz \cite{GK},
several expansion bounds for polynomials over arbitrary fields have been improved.
Barak, Impagliazzo, and Wigderson \cite{BIW} had proved that $f=xy+z$ is an expander over any prime field $\Fp$, with an unspecified $\alpha>1$.
Roche-Newton, Rudnev, and Shkredov \cite{RRS} used \cite{R} to improve the exponent to $\alpha=3/2$ with $\beta=2/3$, over any field $\F$.
In other words, they proved
\[|AB+C|\gg N^{3/2}\]
for $A,B,C\subset\F$ with $|A|=|B|=|C|=N\ll p^{2/3}$.
Aksoy-Yazici et al. \cite{AMRS} proved the same for $f=x(y+z)$. 
There are similar results for expanders in more than three variables, 
but establishing two-variable expanders over finite fields seems to be considerably harder.
Essentially the only known example is $f(x,y)=x^2+xy$, which Bourgain \cite{B} proved to be an expander; Hegyv\'ari and Hennecart \cite{HH} generalized this to polynomials of the form $f(x)+x^kg(y)$ (with certain exceptions). 
Stevens and De Zeeuw \cite{SZ} improved the exponent for $x^2+xy$ to $\alpha=5/4$ with $\beta=2/3$, again using \cite{R}.

Over $\R$, expanders are better understood.
Elekes and R\'onyai \cite{ER} discovered that over $\R$ the two-variable expanders are exactly those polynomials $f(x,y)\in \R[x,y]$ that do not have the additive form $g(h(x)+k(y))$ or the multiplicative form $g(h(x)k(y))$.
Raz, Sharir, and Solymosi \cite{RSS} improved the exponent to $\alpha = 4/3$.
For three-variable polynomials, 
Schwartz, Solymosi, and De Zeeuw \cite{SSZ} 
proved that the only non-expanders over $\R$ have the form $g(h(x)+k(y)+l(z))$ or $g(h(x)k(y)l(z))$, 
and Raz, Sharir, and De Zeeuw \cite{RSZ} proved a quantitative version with $\alpha = 3/2$.

It is natural to conjecture that the same classification of expanders holds over arbitrary fields.
Bukh and Tsimerman \cite{BT} and Tao \cite{T} proved  results in this direction for two-variable polynomials on large subsets of finite fields,
but in general the expander question remains open for two-variable polynomials.
We use the result of Rudnev \cite{R} to make a first step towards classifying three-variable expanders over arbitrary fields, by determining which \emph{quadratic} polynomials are expanders.
The expanders $xy+z$ and $x(y+z)$, mentioned above, 
are special cases.
Note that for quadratic polynomials the exceptional form $g(h(x)k(y)l(z))$ does not occur (if the polynomial depends on each variable).

\begin{theorem}\label{thm:mainexpanding}
Let $f\in \F[x,y,z]$ be a quadratic polynomial that depends on each variable and that does not have the form $g(h(x)+k(y)+l(z))$.
Let $A, B,C\subset \F$ with $|A|=|B|=|C|=N$.
Then
\[|f(A\times B\times C)| \gg \min\left\{N^{3/2},p\right\}.\]
\end{theorem}
In terms of our definition, 
this theorem says that if a quadratic $f\in \F[x,y,z]$ does not have the multiplicative form $g(h(x)+k(y)+l(z))$, 
then it is an expander 
with $\alpha = 3/2$ and $\beta = 2/3$.
The theorem also gives expansion for $2/3<\beta<1$, with $\alpha$ shrinking as $\beta$ approaches $1$.

\paragraph{Consequences.}
One new expander included in our theorem is $f(x,y,z) = (x-y)^2+z$;
all our applications rely on this special case of our main theorem.

We will show that we can use this expander to obtain a new bound on the expression $|A+A^2|$.
This expression was first considered by Elekes, Nathanson, and Ruzsa \cite{ENR},
who observed that it has an expansion-like property, even though $f(x,y) = x+y^2$ is not an expander in the definition above (one could call it a ``weak expander'').
They showed that $|A+A^2|\gg |A|^{5/4}$ for $A\subset \R$,
and the exponent was improved by Li and Roche-Newton \cite{LR} to $24/19$ (up to a logarithmic factor in the bound).
For $A\subset \F_p$, Hart, Li, and Shen \cite{HLS} proved that $|A+A^2|\gg|A|^{147/146}$ for $|A|\ll p^{1/2}$,
which was improved by Aksoy Yazici et al. \cite{AMRS} to 
$|A+A^2|\gg |A|^{11/10}$ for $|A|\ll p^{5/8}$.
Here we improve this bound further.

\begin{theorem}\label{thm:A+A^2}
For $A\subset \mathbb{F}$ with $|A|\ll p^{5/8}$ we have 
\[|A+A^2|\gg |A|^{6/5}.\]
\end{theorem}

A closely related expression is $|A^2+A^2|$,
for which there are expansion-like bounds that are conditional on $|A+A|$ being small. 
Over $\R$, 
 \cite{ENR} proved $\max\left\{|A+A|, |A^2+A^2|\right\}\gg |A|^{5/4}$,
and the exponent was improved to $24/19$ in \cite{LR} (up to logarithms).
Over $\F_p$, \cite{BT} proved a quantitatively weaker version,
and \cite{AMRS} proved that 
$\max\left\{|A+A|, |A^2+A^2|\right\}\gg |A|^{8/7}$ for $|A|\ll p^{3/5}$.
We improve this bound as well.

\begin{theorem}\label{thm:A^2+A^2}
For $A\subset \mathbb{F}$ with $|A|\ll p^{5/8}$ we have 
\[\max\left\{|A+A|, |A^2+A^2|\right\}\gg |A|^{6/5}.\]
\end{theorem}

It is worth noting that the bounds in Theorems \ref{thm:A+A^2} and \ref{thm:A^2+A^2}
are numerically the same as the best known bounds for $\max\{|A+A|,|A\cdot A|\}$ \cite{RRS} and $|A\cdot (A+1)|$ \cite{SZ};
in each case the lower bound is $|A|^{6/5}$ under the condition $|A|\ll p^{5/8}$.

Our expansion bound for $f(x,y,z) = (x-y)^2+z$ also allows us to give inductive proofs of expansion bounds for the algebraic distance function in any number of variables.
This idea is due to Hieu and Vinh \cite{HV} and Vinh \cite{V}, who used it to prove expansion bounds on large subsets of finite fields.
Given $P\subset \F^d$, we define its \emph{distance set} by
\[\Delta(P) = \left\{\sum_{i=1}^d(x_i-y_i)^2: (x_1,\ldots, x_d),(y_1,\ldots,y_d)\in P\right\}\]
Obtaining good expansion bounds for this function in $\R^2$ is the well-known \emph{distinct distance problem} of Erd\H os \cite{E}, 
which is a central question in combinatorial geometry.
Here we prove a general bound for the number of distinct distances determined by a higher-dimensional Cartesian product.
Note that as $d$ increases this bound converges to $|A|^2$ (up to constants).

\begin{theorem}\label{thm:maindistances}
For $A\subset \F$ we have 
\[\left|\Delta(A^d)\right|
\gg  \min\left\{|A|^{2-\frac{1}{2^{d-1}}},p\right\}.\]
\end{theorem}

For $d=2$ we recover the result of Petridis \cite{P} 
that $|\Delta(A^2)|\gg \min\{|A|^{3/2},p\}$, 
which is the current best bound for distinct distances on Cartesian products over general fields.
For large subsets of prime fields, 
we recover a result of Hieu and Vinh \cite{HV}.

Finally, 
we consider $\F=\R$, 
which is of course the field for which Erd\H os \cite{E} introduced the distinct distances problem.
He observed that for $A = \{1,\ldots, N\}$ we have 
$|\Delta(A^2)|\ll |A|^2/\sqrt{\log|A|}$
and
$|\Delta(A^d)|\ll |A|^2 =(|A^d|)^{2/d}$ for $d\geq 3$.
He later conjectured 
that these bounds are optimal for arbitrary point sets,
i.e., that $|\Delta(P)|\gg |P|/\sqrt{\log |P|}$ for all $P\subset \R^2$,
and $|\Delta(P)|\gg |P|^{2/d}$ for all $P\subset \R^d$ with $d\geq 3$.
Guth and Katz \cite{GK} almost solved this for $d=2$, by proving that 
\begin{equation}\label{eq:GK}
|\Delta(P)|\gg |P|/\log |P|
\end{equation}
for any $P\subset \R^2$.
For $d\geq 3$, the best lower bound is due to Solymosi and Vu \cite{SV}.
It is roughly speaking of the form 
\[|\Delta(P)|\gg |P|^{\frac{2}{d} - \frac{1}{d^2}};\]
see Sheffer \cite{S} for the exact expression (incorporating \cite{GK}). 

It follows from \cite{GK} that for any $A\subset\R$ we have $|\Delta(A^d)|\gg|A|^2/\log|A|$, since the set $(A-A)^2+(A-A)^2$ is contained in any set of the form $(A-A)^2+\cdots +(A-A)^2$.
By taking the distinct distance bound of \cite{GK} as the base case for the inductive argument with $(x-y)^2+z$ that we used to prove Theorem \ref{thm:maindistances}, 
we obtain an improvement on the exponent of the logarithm.

\begin{theorem}\label{thm:realdistances}
For $A\subset \R$ and $d\geq 2$ we have 
\[\left|\Delta(A^d)\right|
\gg \frac{|A|^2}{\log^{1/2^{d-2}}|A|}.\]
\end{theorem}

We note that this theorem can also be proved without Rudnev's theorem \cite{R},
by using only the Szemer\'edi--Trotter theorem \cite{ST} and the Guth--Katz bound \cite{GK}; see Section \ref{sec:conseqs}.

\section{Three-variable expanding polynomials}

Our main tool is a point-plane incidence bound of Rudnev \cite{R}.
We use the following slightly strengthened version proved by De Zeeuw \cite{Z} (and our proof relies on this strengthening).
We write $\I(\RR,\S) = |\{(r,s)\in \RR\times \S : r\in s\}|$ for the number of \emph{incidences} of $\RR$ and $\S$.

\begin{theorem}[Rudnev]\label{thm:rudnev}
Let $\RR$ be a set of points in $\F^3$ and let $\S$ be a set of planes in $\mathbb{F}^3$, with $|\RR|\ll |\S|$ and $|\RR|\ll p^2$.
Suppose that there is no line that contains $k$ points of $\RR$ and is contained in $k$ planes of $\S$.
Then
\[ \I(\RR,\S)\ll |\RR|^{1/2}|\S| +k|\S|.\] 
\end{theorem}

To prove Theorem \ref{thm:mainexpanding}, 
we divide the quadratic polynomials into two types: those with only one or two of the mixed terms $xy,xz,yz$, 
and those with all three.
Our approach to both types is similar, but it appears technically simpler to treat these types separately.

\begin{lemma}\label{lem:expandingwithoutyz}
Consider a polynomial
\[f(x,y,z) = axy+ bxz + r(x) + s(y)+t(z),\] for polynomials $r,s,t\in \F[u]$ of degree at most two, with $a\neq0$ and $t(z)$ not constant.
Let $A, B,C\subset \F$  with $|A|,|B|\leq |C|$.
Then
\[|f(A\times B\times C)| \gg \min\left\{|A|^{1/2}|B|^{1/2}|C|^{1/2},p\right\}. \]
\end{lemma}
\begin{proof}
We may assume  $|A||B||C|\ll p^2$.
Otherwise, we can remove elements from the sets, while preserving $|A|,|B|\leq |C|$,
until we have sets $A',B',C'$ that do satisfy
$|A'||B'||C'|\ll p^2$.
The proof below then gives 
$|f(A'\times B'\times C')|\gg |A'|^{1/2}|B'|^{1/2}|C'|^{1/2}=p$.

We let $E$ be the number of solutions $(x,y,z,x',y',z')\in (A\times B\times C)^2$ of 
\[ f(x,y,z) = f(x',y',z').\]
We can rewrite this equation to
\[ ayx- ax'y' +(bxz +r(x) + t(z)- s(y'))
= bx'z'+r(x')+t(z') - s(y).\]
We define a point set
\[\RR = \{(x, y', bxz +r(x)+ t(z)- s(y')) : (x,y',z)\in A\times B\times C \} \]
and a plane set
\[ \S=\{ay X -ax'Y  +Z = bx'z'+r(x')+t(z') - s(y) :
(x',y,z')\in A\times B\times C
\}.\]

A point in $\RR$ corresponds to at most two points $(x,y',z)\in A\times B\times C$, 
since $x$ and $y'$ are determined by the first two coordinates, 
and $z$ is then determined with multiplicity at most two by the quadratic expression in the third coordinate.
Here we use the assumption that $t(z)$ is not constant; the only exception occurs when $t(z)$ is linear and its main term is cancelled out by $bxz$; this is negligible since it only occurs for one value of $x$.
The same argument shows that a plane in $\S$ corresponds to at most two points $(x',y,z')\in A\times B\times C$.
Thus we have $|\RR|,|\S|\approx |A||B||C|$.

A solution of $f(x,y,z)=f(x',y',z')$ corresponds to an incidence between a point in $\RR$ and a plane in $\S$.
Conversely, an incidence corresponds to at most four solutions, since the point and the plane have multiplicity at most two.
Hence $\I(\RR,\S) \approx E$.
By assumption we have $|\RR|\approx |A||B||C| \ll p^2$,
which allows us to apply Theorem \ref{thm:rudnev}.
We need to prove an upper bound on the $k$ such that there is a line containing $k$ points of $\RR$ and contained in $k$ planes of $\S$.

The projection of $\RR$ to the first two coordinates is $A\times B$, 
so a line contains at most $\max\{|A|,|B|\}$ points of $\RR$, unless it is vertical, in which case it could contain $|C|$ points of $\RR$. 
However, the planes in $\S$ contain no vertical lines (since they are defined by equations in which the coefficient of $Z$ is non-zero), so in this case the condition of Theorem \ref{thm:rudnev} holds with $k = \max\{|A|,|B|\}\leq |A|^{1/2}|B|^{1/2}|C|^{1/2}$.

Thus we get
\[ E \approx \I(\RR,\S) 
\ll |A|^{3/2}|B|^{3/2}|C|^{3/2}.\]
By the Cauchy-Schwartz inequality we have $|A|^2|B|^2|C|^2\leq E\cdot |f(A\times B\times C)|$,
so we get
\[|f(A\times B\times C)| \gg |A|^{1/2}|B|^{1/2}|C|^{1/2}.\]
This finishes the proof.
\end{proof}

It would not be hard to generalize Lemma \ref{lem:expandingwithoutyz} to polynomials of the form 
\[f(x,y,z) = g(x)h(y)+ k(x)l(z) + r(x) + s(y)+t(z),\]
with the resulting bound depending on the degrees of $g,h,k,l,r,s,t$.

\begin{lemma}\label{lem:expandingwithyz}
Let $f\in \F[x,y,z]$ be a polynomial of the form
\[f(x,y,z) = axy + bxz + cyz + dx^2+e y^2 + g z^2,\]
with none of $a,b,c$ zero,
and with $4eg\neq c^2$.
Let $A, B,C\subset \F$ with $|A|=|B|=|C|=N$.
Then
\[|f(A\times B\times C)| \gg \min\left\{
N^{3/2},p\right\}. \]
\end{lemma}
\begin{proof}
We may assume  $|A||B||C|\ll p^2$ as in the proof of Lemma \ref{lem:expandingwithoutyz}.
We again bound the number $E$ of solutions $(x,y,z,x',y',z')\in (A\times B\times C)^2$ of $f(x,y,z) = f(x',y',z')$.
We rewrite the equation to
\[ (ay+bz)x  -x' (ay'+bz')+  (d x^2 - (e(y')^2 +cy'z'+ g(z')^2) = 
d(x')^2-(e y^2 + cyz +g z^2).\]
We define a point set
\[\RR = \{(x, ay'+bz', 
 dx^2 - (e (y')^2+cy'z'+ g(z')^2) : (x,y',z')\in A\times B\times C \} \]
and a plane set
\[ \S=\{ (ay+bz)X  -x' Y+ Z = 
d (x')^2-(e y^2 +cyz+ g z^2) :
(x',y,z)\in A\times B\times C
\}.\]

We show that a point $(u,v,w)\in \RR$ corresponds to at most two points $(x,y',z')\in A\times B\times C$.
Suppose that we have $u=x, v = ay'+bz', w = d x^2 -e (y')^2- cy'z'-g (z')^2$.
Then
\[w = d u^2 - e (y')^2 -cy' \frac{v-ay'}{b} - g\left(\frac{v-ay'}{b}\right)^2, \]
or equivalently 
\[\left(b^2d-abc+a^2g\right)(y')^2  +\left(bcv - 2agv\right)y' +b^2w - b^2d u^2+g v^2 = 0.\]
We do not have both $b^2d-abc+a^2g =0$ and $bc-2ag=0$, 
since these two equations would imply $4eg=c^2$, contradicting the assumption of the lemma.
Hence there are at most two values of $y'$ corresponding to $(u,v,w)$, with unique corresponding $x,z'$.

The same argument shows that a plane in $\S$ corresponds to at most two points $(x',y,z)$.
Hence we have $|\RR|,|\S|\approx |A||B||C|$ and $\I(\RR,\S) \approx E$.
By the assumption at the start of the proof we have $|\RR|\approx |A||B||C| \ll p^2$.
This allows us to apply Theorem \ref{thm:rudnev}, if we find an upper bound on the maximum number of collinear points in $\RR$.

The point set $\RR$ is covered by $|A|$ planes of the form $x=x_0$. 
If a line is not in one of these planes, 
then it intersects $\RR$ in at most $|A|=N$ points.
Let $\ell$ be a line contained in a plane $x=x_0$.
The points of $\RR$ in this plane lie on a curve which is either a parabola or a line.
In the first case, 
$\ell$ intersects the parabola in at most two points.
In the second case, $\ell$ either intersects the line in one point, or it equals that line, which contains $|C|$ points.
It is easy to see from the equations that for distinct $y'$ we get distinct curves, so the case where the curve equals $\ell$ occurs at most once.
This implies that $\ell$ contains at most $2|B|+|C|\ll N$ points of $\RR$.

With $k=N$ we get $E\approx\I(\RR,\S) 
\ll (N^3)^{3/2} + N\cdot N^3 \ll N^{9/2}$,
and again using Cauchy-Schwartz we get $|f(A\times B\times C)| \gg N^{3/2}$.
This finishes the proof.
\end{proof}

We now combine the two lemmas to prove Theorem \ref{thm:mainexpanding}.

\begin{proof}[Proof of Theorem \ref{thm:mainexpanding}]
Let $f(x,y,z)$ be a quadratic polynomial that is not of the form $g(h(x)+k(y)+l(z))$.
In particular, $f$ has at least one of the mixed terms $xy,xz,yz$, 
since otherwise it would be of the form $h(x)+k(y)+l(z)$.
If one of the terms $xy,xz,yz$ does not occur in $f$, 
then Lemma \ref{lem:expandingwithoutyz} proves the theorem.

Thus we can assume that $f$ has the form 
\[f(x,y,z) = axy+bxz+cyz + r(x)+s(y)+t(z),\]
with $a,b,c$ non-zero
and $r,s,t$ polynomials of degree at most two.
We may assume that $r,s,t$ have no constant or linear terms.
Indeed, any constant term can be removed immediately, 
and any linear terms can be removed by a change of variables of the form $\widetilde{x} = p_1x+q_1,\widetilde{y}=p_2y+q_2,\widetilde{z}=p_3z+q_3$.
Thus we assume that $f$ has the form
\[f(x,y,z) = axy + bxz + cyz + d x^2+ey^2 + g z^2.\]
The assumption that $f$ is not of the form $g(h(x)+k(y)+l(z))$, which still holds after the linear change of variables, implies that the equations $4de=a^2, 4dg=b^2,4eg=c^2$ do not all hold. 
Otherwise, we could write $f=(\sqrt{d}x+\sqrt{e}y+\sqrt{g}z)^2$ (if $d,e,g$ are not squares in $\F$, we can write $f=(d\sqrt{eg}x+e\sqrt{dg}y+g\sqrt{de}z)^2/deg$).
By permuting the variables, we can assume that $4eg\neq c^2$.
Then we can apply Theorem \ref{lem:expandingwithyz},
which finishes the proof.
\end{proof}

\section{Consequences of Theorem \ref{thm:mainexpanding}}
\label{sec:conseqs}

\begin{proof}[Proof of Theorem \ref{thm:A+A^2}]
We consider the equation
\begin{equation}\label{eq:eq1}(x-y)^2+z=t.
\end{equation}
Observe that for any $a,b,c\in A$, a solution of \eqref{eq:eq1} is given by $x =a+b^2\in A+A^2$, $y = b^2\in A^2$, $z = c\in A$, and $t = c+a^2\in A+A^2$.
Thus we have
\begin{equation}\label{eq:eq2}
|A|^3\leq \left|\left\{(x,y,z,t)\in (A+A^2)\times A^2\times A\times (A+ A^2): (x-y)^2+z=t\right\}\right|.
\end{equation}
If we set 
\[E = \left| \left\lbrace (x,y,z,x',y',z')\in ((A+A^2)\times A^2\times A)^2\colon (x-y)^2+z=(x'-y')^2+z'\right\rbrace\right|, \]
then \eqref{eq:eq2} and the Cauchy-Schwarz inequality give
\begin{equation}\label{eq:eq3}
\frac{|A|^6}{|A+A^2|}\leq E.
\end{equation}

We now partly follow the proof of Lemma \ref{lem:expandingwithoutyz} for $f(x,y,z)=(x-y)^2+z$.
We define a point set
\[ \mathcal{R} = \{(x,y', x^2+z-(y')^2): (x,y',z)\in (A+A^2)\times A^2\times A\}\]
and a plane set 
\[ \mathcal{S} = \{-2yX + 2x'Y + Z = (x')^2+z' - y^2: (x',y,z')\in (A+A^2)\times A^2\times A\}.\]
We are already done if $|A+A^2|\gg |A|^{6/5}$, so we can assume that $|A+A^2|\ll |A|^{6/5}$, which gives $|\mathcal{R}|\approx |A+A^2||A^2||A|\ll |A|^{16/5}\ll p^2$, using the assumption $|A|\ll p^{5/8}$.
Thus we can apply Theorem \ref{thm:rudnev}.
By the same argument as in the proof of Lemma \ref{lem:expandingwithoutyz}, 
we can use $k = \max\{|A+A^2|, |A^2|\} = |A+A^2|$,
so we get
\begin{equation}\label{eq:eq4}
E\ll  |A+A^2|^{3/2}|A|^3 + |A+A^2|^2|A|^2.
\end{equation}
If the second term is larger than the first, then we have $|A+A^2| \gg |A|^2$, and we would be done.
Otherwise, the first term is larger,
so combining \eqref{eq:eq3} and \eqref{eq:eq4} gives 
\[\frac{|A|^6}{|A+A^2|} \ll |A+A^2|^{3/2}|A|^3,\]
which implies that 
\[|A+A^2|\gg |A|^{6/5}.\]
This completes the proof of the theorem.
\end{proof}

\begin{proof}[Proof of Theorem \ref{thm:A^2+A^2}]
The proof is very similar to that of Theorem \ref{thm:A+A^2}, and we omit most of the details.
The key observation, analogous to \eqref{eq:eq2}, is
\[|A|^3\leq \left|\left\{(x,y,z,t)\in (A+A)\times A\times A^2\times (A^2+ A^2): (x-y)^2+z=t\right\}\right|. \]
By following the steps in the proof of Theorem \ref{thm:A+A^2} we now obtain
\[ \frac{|A|^6}{|A^2+A^2|} \ll |A+A|^{3/2}|A|^3\]
under the condition $|A|\ll p^{5/8}$, 
which gives
\[ |A+A|^3|A^2+A^2|^2 \gg |A|^6.\]
This proves the theorem.
\end{proof}


\newpage

To prove (a generalization of) Theorem \ref{thm:maindistances}, 
we use a special case of Lemma \ref{lem:expandingwithoutyz}.

\begin{corollary}\label{cor:g(AxA)+C}
Let $g\in \F[x,y]$ be a quadratic polynomial with a non-zero $xy$-term.
Let $A,C\subset \F$ with $|A|\leq |C|$.
Then
\[|g(A\times A)+C|\gg \min\left\{|A||C|^{1/2},p\right\}. \]
\end{corollary}


\begin{theorem}\label{thm:generaldistances}
Let $g_1, \ldots, g_d\in \F[x,y]$ be quadratic polynomials, each of which has a non-zero $xy$-term.
Then for $A\subset \F$ we have
\[\left|\sum_{i=1}^d g_i(A\times A)\right|\gg \min\left\{|A|^{2-\frac{1}{2^{d-1}}},p\right\}.\]
\end{theorem}
\begin{proof}
Set $G_k(x_1,y_1\ldots, x_k,y_k) = \sum_{i=1}^k g_i(x_i,y_i)$.
We prove by induction on $k$ that
\[\left|G_k(A^{2k})\right|\gg \min\left\{|A|^{2-\frac{1}{2^{k-1}}},p\right\}.\]
The base case $k=1$ holds trivially.
Suppose that the claim holds for some $k$ with $1\leq k<d$. 
Applying Corollary \ref{cor:g(AxA)+C} with $g=g_{k+1}$ and $C = G_k(A^{2k})$ gives
\[|G_{k+1}(A^{2(k+1)})|\gg  \min\left\{|A|\left(|A|^{2-\frac{1}{2^{k-1}}}\right)^{1/2},p\right\} = \min\left\{|A|^{2-\frac{1}{2^{(k+1)-1}}},p\right\}.\]
This proves the theorem.
\end{proof}

Theorem \ref{thm:maindistances} follows immediately by setting $g_i=(x-y)^2$ for every $i$ in Theorem \ref{thm:generaldistances}.
To prove Theorem \ref{thm:realdistances},
we merely have to start the induction at $k=2$, and plug in the result of Guth and Katz \cite{GK}.

\begin{proof}[Proof of Theorem \ref{thm:realdistances}.]
Set $\F=\R$.
We prove $|\Delta(A^d)|\gg |A|^2/\log^{1/2^{d-2}}|A|
$ by induction on $d$.
The base case $d=2$ follows from the main result of \cite{GK}, 
stated here as \eqref{eq:GK} in 
Section \ref{sec:intro}.
Suppose that the claim holds for some $d>2$.
Applying Corollary \ref{cor:g(AxA)+C} with $g=(x-y)^2$ and $C = \Delta(A^d)$ gives
\[|\Delta(A^{d+1})|\gg 
|A||\Delta(A^d)|^{1/2}
\gg |A|\left(\frac{|A|^2}{\log^{1/2^{d-2}}|A|}\right)^{1/2} = \frac{|A|^2}{\log^{1/2^{(d+1)-2}}|A|}.\]
This proves the theorem.
\end{proof}

Although this proof arose naturally from our general approach, 
it is worth noting that over $\R$ it is possible to prove the relevant case of Corollary \ref{cor:g(AxA)+C} using only the Szemer\'edi-Trotter theorem \cite{ST}, 
which leads to a proof of Theorem \ref{thm:realdistances} without Theorem \ref{thm:rudnev}.

\begin{proof}[Alternative proof of Theorem \ref{thm:realdistances}.]
We define a point set and curve set by
\[\mathcal{P} = A\times ((A-A)^2+C),~~~~~
\mathcal{C} = \{Y = (X-a)^2+c: (a,c)\in A\times C\}. \]
The curves in $\mathcal{C}$ are parabolas, 
but we can apply the bijection $\varphi:(X,Y)\mapsto (X,Y-X^2)$,
which sends the parabola $Y = X^2-2aX+a^2+c$ to the line $Y' = -2aX'+a^2+c$.
Applying the Szemer\'edi-Trotter theorem \cite{ST} to the points $\varphi(\mathcal{P})$ and the lines $\varphi(\mathcal{C})$ gives
\[|A|^2|C| \leq \I(\varphi(\mathcal{P}),\varphi(\mathcal{C}))
\ll (|A||(A-A)^2+C|)^{2/3}(|A||C|)^{2/3}+|A||(A-A)^2+C| + |A||C|.
\]
It follows that $|(A-A)^2+C| \gg |A||C|^{1/2}$.

We can now prove the theorem by induction exactly as in the previous proof.
\end{proof}

\bigskip
\bigskip

We are finished proving the main theorems in Section \ref{sec:intro},
but we give one more  application that we find interesting.

Another polynomial in the form of Theorem \ref{thm:generaldistances}
is the dot product function.
For $P\subset \F^d$,
define its \emph{dot product set} by
\[\Pi(P) = \left\{\sum_{i=1}^d x_iy_i: (x_1,\ldots, x_d),(y_1,\ldots,y_d)\in P\right\}.\]
Choosing $g_i=xy$ for every $i$ in Theorem \ref{thm:generaldistances} gives $|\Pi(A^d)|
\gg  \min\{|A|^{2-\frac{1}{2^{d-1}}},p\}$ for $A\subset \F$.
This bound was proved for $d=2,3$ in \cite{RRS}.
More interestingly, 
we can prove that a better expansion bound holds for distances \emph{or} for dot products (or for both).

\begin{theorem}\label{thm:distanceordot}
For $A\subset \F$ with $|A|\ll p^{\frac{1}{2}+\frac{1}{5\cdot 2^{d-1}-2}}$ we have 
\[\max\left\{|\Pi(A^{d})|, 
|\Delta(A^{d})|\right\} \gg |A|^{2-\frac{1}{5\cdot 2^{d-3}}}.\]
\end{theorem}
\begin{proof}
We prove the theorem by induction on $d$.
For $d=1$,
we have $|\Delta(A)| \gg |A-A|$,
so the statement follows from the sum-product bound 
\[ \max\left\{|A-A|,|A\cdot A|\right\}\gg |A|^{6/5},\]
which was proved 
in \cite{RRS} (also as a consequence of \cite{R}).
Assume that the claim holds for $d>1$.
If $|\Delta(A^d)|\ge |\Pi(A^d)|$,
then we set $g = (x-y)^2$ and  $C=\Delta(A^d)$, 
so that Corollary \ref{cor:g(AxA)+C} gives
\[|\Delta(A^{d+1})|
= |g(A\times A)+ \Delta(A^d)|
\gg |A|\left(|A|^{2-\frac{1}{5\cdot 2^{d-3}}}\right)^{1/2}
= |A|^{2-\frac{1}{5\cdot 2^{(d+1)-3}}}.\]
If $|\Pi(A^d)|\ge |\Delta(A^d)|$,
then we set $g = xy$ and  $C=\Pi(A^d)$,
and do the same calculation.
\end{proof}

An analogous bound for large subsets of finite fields was proved in \cite{V}.


\end{document}